\documentclass[reqno,12pt]{amsart}
\usepackage{amsmath,amsthm,amsfonts,amssymb}
\usepackage[cp1251]{inputenc}
   \usepackage[english]{babel}

\date{} 
\newtheorem{theorem}{Theorem}[section]

\newtheorem{proposition}[theorem]{Proposition}
\newtheorem{remark}[theorem]{Remark}
\newtheorem{corollary}[theorem]{Corollary}
\newtheorem{example}[theorem]{Example}
\numberwithin{equation}{section}

\begin{document}

\centerline{{\bf Optimal transportation of measures with a parameter}}

\vskip .2in

\centerline{{\bf Vladimir I. Bogachev$^{a,b,}$\footnote{Corresponding author, vibogach@mail.ru.},
Svetlana N. Popova$^{b,c}$}}

\vskip .2in

$^{a}$ Department of Mechanics and Mathematics,
Moscow State University,
119991 Moscow, Russia

$^{b}$  National Research University Higher School of Economics, Moscow, Russia

$^{c}$ Moscow Institute of Physics and Technology, Dolgoprudny, Moscow Region, Russia

\vskip .2in

{\bf Abstract.}
We consider optimal transportation of measures on metric and topological spaces
in the case where the cost function and marginal distributions depend on a parameter with
values in a metric space. The Hausdorff distance
between the sets of probability measures with given marginals is  estimated via the distances  between the marginals themselves.
 The continuity of the cost of optimal transportation with respect to the parameter is proved in the case of  continuous
dependence of the cost function and marginal distributions on this parameter.
Existence of approximate optimal plans continuous with respect to the parameter is established.
It is shown that the optimal plan is continuous with respect to the parameter in case of uniqueness.
Examples are constructed when there is no continuous selection of optimal plans.
Finally, a general result on convergence of Monge optimal mappings is proved.

\vskip .1in


\vskip .1in


\section{Introduction}

Let us recall that the Kantorovich optimal transportation problem deals with a triple $(\mu,\nu, h)$, where $\mu$ and
$\nu$ are Borel probability measures on topological spaces $X$ and $Y$, respectively, and $h\ge 0$ is a Borel function
on~$X\times Y$. The problem concerns minimization of  the integral
$$
\int h\, d\sigma
$$
over all measures $\sigma$ in the set $\Pi(\mu,\nu)$ consisting of Borel probability measures on $X\times Y$
with projections $\mu$ and $\nu$ on the factors, that is, $\sigma (A\times Y)=\mu(A)$ and $\sigma (X\times B)=\nu(B)$
for all Borel sets $A\subset X$ and $B\subset Y$. The measures $\mu$ and $\nu$ are called marginal distributions
or marginals, and $h$ is called a cost function.
In general, there is only infimum $K_h(\mu,\nu)$ (which may be infinite),
but if $h$ is continuous (or lower semicontinuous)
and bounded and $\mu$ and $\nu$ are Radon, then the minimum is attained and measures on which it is attained are called
optimal measures or optimal Kantorovich plans. Moreover, the boundedness of $h$ can be replaced by the assumption that
there is a measure in $\Pi(\mu,\nu)$ with respect to which $h$ is integrable. The problem is also meaningful in the purely
set-theoretic setting, but here we consider the topological case, moreover, the spaces under consideration will be
completely regular, in some results metric. Similarly the multimarginal Kantorovich problem is introduced with
marginals $\mu_1, \ldots,\mu_n$ on spaces $X_1,\ldots,X_n$.
General information about Monge and Kantorovich problems can be found
in \cite{AG}, \cite{B18}, \cite{BK}, \cite{RR}, \cite{Sant}, \cite{V1}, and~\cite{V2}.

We study optimal transportation of measures on metric and topological spaces
in the case where the cost function $h_t$ and marginal distributions $\mu_t$ and $\nu_t$  depend on a parameter $t$ with
values in a metric space. Here the questions naturally arise about the continuity with respect to $t$ of
the optimal cost $K_{h_t}(\mu_t,\nu_t)$ and also about the possibility to select an  optimal plan in $\Pi(\mu_t,\nu_t)$ continuous with respect to
the parameter. In addition, the set of all transport plans $\Pi(\mu_t,\nu_t)$ also depends on the parameter,
so that one can ask about its continuity when the space of sets of measures is equipped with the Hausdorff metric
generated by some metric on the space of measures.
Kantorovich problems depending on a parameter were investigated in several papers, see
\cite{V2}, \cite{Z}, \cite{Mal}, \cite{KNS}, \cite{BM},   and~\cite{BDM},  where the questions of measurability were addressed.
The continuity properties are also of great interest, in particular, they can be useful in the study of
differential equations and inclusions on spaces of measures (see~\cite{BF}),
in regularization of optimal transportation  (see \cite{C21} and~\cite{LMM}), and in other applications.

The main results of the paper are these.

1. The Hausdorff distance
between the sets of probability measures $\Pi(\mu_1,\nu_1)$ and $\Pi(\mu_2,\nu_2)$ is  estimated via the distances  between
 $\mu_1$ and $\mu_2$ and $\nu_1$ and~$\nu_2$ (see Theorem~\ref{maint1}). An analogous result holds for the $p$-Kantorovich
 metric~$W_p$.

2. The cost of optimal transportation is continuous  with respect to the parameter in the case of  continuous
dependence of the cost function and marginal distributions on this parameter (see Theorem~\ref{maint2}).

3. The optimal plan is continuous with respect to the parameter in case of uniqueness (see Corollary~\ref{conv-pl}).

4. There exist approximate optimal plans continuous with respect to the parameter (see Theorem~\ref{t-appr}).

5. Examples are constructed when there is no continuous selection of optimal plans on $[0,1]^2$ with a cost function continuous
with respect to the parameter and marginal distributions equal to Lebesgue measure.

Finally, we prove a general result on convergence of Monge optimal mappings in the spirit
of known results from \cite{AP} and~\cite{DPLS}: according to Corollary~\ref{conv-meas},
Monge optimal mappings taking $\mu_n$ to $\nu_n$ converge in measure $\mu_0$ if $\mu_n\to\mu_0$ in variation,
$\nu_n\to\nu_0$ weakly, the cost functions $h_n$ converge to the cost function $h_0$ uniformly on compact sets,
 and the corresponding optimal Monge mappings are unique.

\section{Notation and terminology}

We recall that a nonnegative Radon measure on a topological space $X$ is a bounded Borel measure $\mu\ge 0$
such that for every Borel set $B$ and every $\varepsilon>0$ there is a compact set $K\subset B$ such that
$\mu(B\backslash K)<\varepsilon$ (see \cite{B07}). If $X$ is a complete separable metric space,
then all Borel measures are Radon.

The space $\mathcal{M}_r(X)$ of signed bounded Radon measures on $X$ can be equipped with
the weak topology generated on the seminorms
$$
\mu\mapsto \biggl|\int f\, d\mu\biggr|,
$$
where $f$ is a bounded continuous function.

A set $\mathcal{M}$ of nonnegative Radon measures on a  space $X$ is called uniformly tight,
if for every $\varepsilon>0$ there exists a  compact set $K\subset X$ such that
$\mu(X\backslash K)<\varepsilon$ for all $\mu\in\mathcal{M}$.

Let $(X, d_X)$ and $(Y, d_Y)$ be metric spaces. The space $X\times Y$ is equipped with the metric
$$
d((x_1,y_1),(x_2,y_2))= d_X(x_1,x_2)+d_Y(y_1,y_2).
$$
The weak topology on the spaces of Radon probability measures $\mathcal{P}_r(X)$, $\mathcal{P}_r(Y)$,  $\mathcal{P}_r(X\times Y)$
is metrizable by  the  corresponding Kantorovich--Rubinshtein metrics  $d_{KR}$ (also called the
Fortet--Mourier  metrics, see \cite{B18}) defined by
$$
d_{KR}(\mu,\nu)=\sup \biggl\{\int f\, d(\mu - \nu)\colon f\in {\rm Lip}_1, \ |f|\le 1\biggr\},
$$
where ${\rm Lip}_1$ is the space of $1$-Lipschitz functions. If $X$ is complete, then
$(\mathcal{P}_r(X), d_{KR})$ is also complete and if $X$ is Polish, then $\mathcal{P}_r(X)$ is also Polish.

The subsets $\mathcal{P}_r^1(X)$, $\mathcal{P}_r^1(Y)$,  $\mathcal{P}_r^1(X\times Y)$, consisting of measures
with respect to which  all functions of the form $x\mapsto d(x,x_0)$ are integrable,
are equipped with the Kantorovich metric
$$
d_{K}(\mu,\nu)=\sup \biggl\{\int f\, d(\mu - \nu)\colon f\in {\rm Lip}_1 \biggr\}.
$$
This metric is also called the Wasserstein metric, but we do not use this historically incorrect terminology.

Since we consider probability measures, in the formula for $d_{K}$ the supremum can be evaluated over functions
$f$ with the additional condition $f(x_0)=0$ for a fixed point~$x_0$. Hence for a space contained in a ball
of radius~$1$  the equality $d_K=d_{KR}$ holds.

Note that an unbounded metric space $(X,d)$ can be equipped with the bounded metric
$\widetilde{d}=\min(d,1)$ generating the original topology. For the new metric we have $\widetilde{d}_K=\widetilde{d}_{KR}$.
Moreover,
\begin{equation}\label{est-two}
2^{-1}\widetilde{d}_K \le d_{KR} \le 2\widetilde{d}_K.
\end{equation}
Indeed, if $|f|\le 1$ and $f$ is  $1$-Lipschitz in the original metric, then with respect to the new metric $f$
is $2$-Lipschitz. On the other hand, every $1$-Lipschitz function $f$ in the new metric vanishing at the
point~$x_0$ satisfies the bound  $|f|\le 1$ and is $2$-Lipschitz in the original metric, since
$\widetilde{d}(x,y)= d(x,y)$ whenever $d(x,y)\le 1$, and when $d(x,y)>1$ the desired inequality follows from the estimate $|f|\le 1$.

It is worth noting that if $X=Y$ and we take the distance as a cost function, then
the equality
$K_d(\mu,\nu)=d_K(\mu,\nu)$
holds on $\mathcal{P}_r^1(X)$, which is called the Kantorovich duality formula.

Similarly, for any $p\in [1,+\infty)$ the subspace $\mathcal{P}_r^p(X)$ in $\mathcal{P}_r(X)$
consisting of measures with respect to which the function $x\mapsto d(x,x_0)^p$ is integrable for some
(and then for all) fixed point~$x_0$, can be equipped with the $p$-Kantorovich metric
$$
W_p(\mu,\nu)=K_{d^p}(\mu, \nu)^{1/p}.
$$

Let us recall that the Hausdorff distance between bounded closed subsets $A$ and $B$ of a metric space $(M,d)$
is defined by the formula
$$
H(A,B)=\max \Bigl\{\sup_{x\in A} d(x,B), \sup_{y\in B} d(y,A)\}.
$$
This distance will be considered  for subsets of the space of
probability measures $\mathcal{P}_r(X\times Y)$ with the Kantorovich--Rubinshtein metric $d_{KR}$
(generated by the metric on $X\times Y$ introduced above)
or its subspace $\mathcal{P}_r^1(X\times Y)$ with the Kantorovich  metric $d_{K}$,
which gives the corresponding Hausdorff distances $H_{KR}$ and $H_K$.
Due to estimates (\ref{est-two}) we can deal with the latter case.

When $\mathcal{P}_r^p(X\times Y)$ is equipped with the metric $W_p$, we obtain the Hausdorff distance
$H_K^p$ on the space of closed subsets of $\mathcal{P}_r^p(X\times Y)$.

Similar constructions are introduced in a more general case of completely regular spaces $X$ and $Y$.
The topologies of such spaces can be defined by families of pseudometrics  $\Psi_X$ and $\Psi_Y$
(recall that a  pseudometric $d$ differs from a metric by the property that the equality $d(x,y)=0$ is allowed
for $x\not=y$).
Then the topology of the product $X\times Y$ is generated by the pseudometrics
$$
((x_1,y_1), (x_2,y_2))\mapsto d_1(x_1,y_1)+d_2(x_2,y_2), \quad d_1\in \Psi_X, d_2\in \Psi_Y.
$$
For a pseudometric $d\in \Psi_X$, in the way indicated above one defines the
Kantorovich and Kantorovich--Rubinshtein pseudometrics $d_{K,d}$ and $d_{KR,d}$ on the spaces $\mathcal{P}_r^\Psi(X)$
and $\mathcal{P}_r(X)$, where the former
consists of measures with respect to which  the functions
$x\mapsto d(x,x_0)$ are integrable for all pseudometrics in $\Psi_X$. It is readily verified that the weak topology on
$\mathcal{P}_r(X)$ is generated by the family of such pseudometrics $d_{KR,d}$.

The Monge problem for the same triple $(\mu, \nu, h)$ is finding a Borel mapping $T\colon X\to Y$ taking $\mu$ into $\nu$,
that is $\nu=\mu\circ T^{-1}$, $(\mu\circ T^{-1})(B)=\mu(T^{-1}(B))$ for all Borel sets $B \subset Y$, for which the integral
$$
\int h(x, T(x))\, \mu(dx)
$$
is minimal. Again, in general there is only the infimum $M_h(\mu, \nu)$ of this integral (possibly, infinite),
but in many interesting cases there exist optimal Monge mappings. In any case,
$K_h(\mu,\nu)\le M_h(\mu, \nu)$, but for non-atomic separable Radon measures and continuous cost functions
one has $K_h(\mu,\nu)= M_h(\mu, \nu)$, see \cite{BKP}, \cite{P}. It follows from this equality that if there is a unique
solution $T$ to the Monge problem, then the image of $\mu$ under the mapping $x\mapsto (x,T(x))$ is an optimal
Kantorovich plan.

We shall need below the so-called ``gluing lemma'',
see, e.g., \cite[Lemma~1.1.6]{BK} or \cite[Lemma~3.3.1]{B18} (in the latter the formulation deals with metric spaces,
   but the proof is actually given for Radon measures on completely regular spaces).
   Let $X_1, X_2, X_3$ be completely regular spaces and let
   $\mu_{1,2}$ and $\mu_{2,3}$ be Radon probability measures on $X_1\times X_2$ and $X_2\times X_3$, respectively,
   such that their projections on $X_2$ coincide.
Then there exists a Radon probability measure $\mu$ on $X_1\times X_2\times X_3$
such  that its projection on $X_1\times X_2$ is $\mu_{1,2}$ and its projection on $X_2\times X_2$ is $\mu_{2,3}$.

\section{Main results}

We start with the following general estimate for measures on completely regular spaces.
For functions $\alpha_1,\ldots,\alpha_n$ on spaces $X_1^2,\ldots, X_n^2$ (where $X_i^2=X_i\times X_i$)
we set
$$
(\alpha_1\oplus \cdots \oplus \alpha_n)((x_1, \ldots, x_n), (x_1', \ldots, x_n'))
=\alpha_1(x_1,x_1')+\cdots+\alpha_n(x_n,x_n'),
$$
where $x_1, x'_1 \in X_1$, \dots, $x_n, x'_n \in X_n$.

\begin{theorem}\label{maint1}
Let $\mu_1, \mu_2\in \mathcal{P}_r(X)$, $\nu\in \mathcal{P}_r(Y)$,
and let $\alpha$ and $\beta$ be continuous nonnegative functions on~$X^2$ and $Y^2$, respectively, where $\beta(y,y)=0$.
Then, for every measure $\sigma_1\in \Pi(\mu_1,\nu)$,  there exists a  measure
$\sigma_2\in \Pi(\mu_2,\nu)$ such that
\begin{equation}\label{est-main-new}
K_{\alpha\oplus \beta}(\sigma_1,\sigma_2)\le K_{\alpha}(\mu_1,\mu_2).
\end{equation}
If $\nu_1, \nu_2\in \mathcal{P}_r(Y)$ and  $\alpha, \beta$ are pseudometrics, then,
for every measure $\sigma_1\in \Pi(\mu_1,\nu_1)$, there exists a measure
$\sigma_2\in \Pi(\mu_2,\nu_2)$ such that
\begin{equation}\label{est-main1-new}
d_{K, \alpha\oplus \beta}(\sigma_1,\sigma_2)\le d_{K, \alpha}(\mu_1,\mu_2)+d_{K, \beta}(\nu_1,\nu_2).
\end{equation}
Hence, for the corresponding Kantorovich and Hausdorff pseudometrics we have
\begin{equation}\label{est-main2-new}
H_{K, \alpha\oplus \beta}(\Pi(\mu_1,\nu_1), \Pi(\mu_2,\nu_2))\le d_{K, \alpha}(\mu_1,\mu_2)+d_{K, \beta}(\nu_1,\nu_2).
\end{equation}
A similar assertion is true for $n$ marginals:
if $\mu_i, \nu_i\in\mathcal{P}_r(X_i)$, $i=1,\ldots,n$, $\alpha_i$ are continuous
pseudometrics on $X_i$,
$\sigma \in \Pi(\mu_1, \ldots, \mu_n)$, then there exists a measure $\pi \in \Pi(\nu_1, \ldots, \nu_n)$ such  that
$$
d_{K, \alpha_1\oplus \cdots\oplus \alpha_n}(\pi, \sigma)
\le d_{K, \alpha_1}(\mu_1,\nu_1)+ \cdots + d_{K, \alpha_n}(\mu_n,\nu_n).
$$
\end{theorem}
\begin{proof}
Let us take  a measure $\eta \in \Pi(\mu_1, \mu_2)$ such that
$$
\int_{X\times X} \alpha(x_1, x_2) \, \eta(dx_1\, dx_2) = K_{\alpha}(\mu_1,\mu_2).
$$
   According to the aforementioned gluing lemma,
there exists a measure $\lambda\in \mathcal{P}_r(X_1\times X_2\times Y)$,
where $X_1=X_2=X$, such  that its projection on $X_1\times Y$ is $\sigma_1$ and its projection on $X_1\times X_2$ is $\eta$.
For $\sigma_2$ we take the projection of the measure $\lambda$ on~$X_2\times Y$. The projections of the measure $\sigma_2$ on $X$ and $Y$ equal
$\mu_2$ and $\nu$, respectively. Indeed, the projection on $X$ is obtained by projecting the measure $\lambda$ first on $X_2\times Y$,
next on~$X_2$, which coincides with the composition of the operators of projecting on $X_1\times X_2$ and $X_2$, but this composition takes
$\lambda$ into $\mu_2$, since $\lambda$ is first mapped to $\eta$ and then to~$\mu_2$. The projection of the measure $\sigma_2$ to  $Y$
is obtained by  projecting first on $X_1\times Y$ and then on~$Y$, that is, equals the projection of the measure $\sigma_1$ on~$Y$.

For the proof of the desired estimate we consider the measure $\zeta$ equal to the image of $\lambda$ under the mapping
$$
X_1\times X_2\times Y\to X_1\times Y_1\times X_2\times Y_2, \quad  Y_1=Y_2=Y, \
  (x_1, x_2,y)\mapsto (x_1, y, x_2, y).
$$
Then $\zeta\in \Pi(\sigma_1,\sigma_2)$, since the projection of the measure $\lambda$ on $X_1\times Y$ is $\sigma_1$ and the
projection on $X_2\times Y$ is $\sigma_2$. Moreover,
\begin{align*}
K_{\alpha\oplus \beta}(\sigma_1,\sigma_2)
&\le
\int_{X_1\times Y_1\times X_2\times Y_2}
 \alpha\oplus \beta\, d\zeta
 \\
 &=\int_{X_1\times X_2\times Y} (\alpha\oplus \beta)((x_1,y), (x_2,y))\, \lambda(dx_1\, dx_2\, dy)
\\
&=
\int_{X_1\times X_2\times Y}  \alpha(x_1,x_2)\, \lambda(dx_1\, dx_2\, dy)
\\
&=
\int_{X_1\times X_2}  \alpha(x_1,x_2)\, \eta(dx_1\, dx_2)= K_{\alpha}(\mu_1,\mu_2).
\end{align*}
In case of pseudometrics we
first pick $\sigma_2\in \Pi(\mu_2,\nu_1)$ satisfying the bound
$$
K_{\alpha\oplus \beta}(\sigma_1,\sigma_2)\le d_K(\mu_1,\mu_2)
$$
and next pick
$\sigma_3\in \Pi(\mu_2,\nu_2)$ satisfying the bound
$$
K_{\alpha\oplus \beta}(\sigma_2,\sigma_3)\le d_K(\nu_1,\nu_2).
$$
It remains to use the triangle inequality
$$
K_{\alpha\oplus \beta}(\sigma_1,\sigma_3)\le K_{\alpha\oplus \beta}(\sigma_1,\sigma_2)+K_{\alpha\oplus \beta}(\sigma_2,\sigma_3).
$$

Let us proceed to the  case of $n$ marginals. Here  it suffices to prove that there  exists a  measure
$\pi \in \Pi(\nu_1, \mu_2, \ldots, \mu_n)$ such that
$$
d_{K, \alpha_1 \oplus \cdots \oplus \alpha_n}(\pi,\sigma) \le d_{K, \alpha_1}(\mu_1,\nu_1).
$$
Let $\sigma_1$ be the projection of $\sigma$ on $X_2 \times \ldots \times X_n$.
Then $\sigma \in \Pi(\mu_1,  \sigma_1)$.
As shown above, there exists a measure $\pi \in \Pi(\nu_1, \sigma_1)$ for which
$d_{K, \alpha_1 \oplus \cdots \oplus \alpha_n}(\pi,\sigma) \le d_{K, \alpha_1}(\mu_1,\nu_1)$.
In addition, the inclusion $\pi \in \Pi(\nu_1, \mu_2, \ldots, \mu_n)$ holds, which completes the proof.
\end{proof}

\begin{remark}
\rm
If in place of the continuity of $\alpha$ and $\beta$ we require only their
Borel measurability  (or universal measurability), then the reasoning used above shows that for every
$\varepsilon>0$ there is a  measure $\sigma_2^\varepsilon$ for which
$$
K_{\alpha\oplus \beta}(\sigma_1,\sigma_2^\varepsilon)\le K_{\alpha}(\mu_1,\mu_2)+\varepsilon.
$$
To  this end we take a measure $\eta^\varepsilon \in \Pi(\mu_1, \mu_2)$ such  that
$$
\int_{X\times X} \alpha(x_1, x_2) \, \eta^\varepsilon(dx_1\, dx_2) \le K_{\alpha}(\mu_1,\mu_2)+\varepsilon,
$$
then  the last equality in the proof with $K_{\alpha}(\mu_1,\mu_2)$ should be replaced with the
inequality with $K_{\alpha}(\mu_1,\mu_2)+\varepsilon$.
\end{remark}

A straightforward modification of the proof yields the following assertion for $W_p$.

\begin{corollary}
Let $p\in [1,+\infty)$ and
$\mu_1, \mu_2, \nu_1, \nu_2\in \mathcal{P}_r^p(X)$.
For every measure $\sigma_1\in \Pi(\mu_1,\nu_1)$, there exists a measure
$\sigma_2\in \Pi(\mu_2,\nu_2)$ such that
$$
W_p(\sigma_1,\sigma_2)\le W_p(\mu_1,\mu_2)+ W_p(\nu_1,\nu_2).
$$
Hence for the Hausdorff distance $H_p$ we have
$$
H_{p}(\Pi(\mu_1,\nu_1), \Pi(\mu_2,\nu_2))\le  W_p(\mu_1,\mu_2)+ W_p(\nu_1,\nu_2).
$$
\end{corollary}
\begin{proof}
Repeating the reasoning given above with $h=d^p$, we first find a measure
$\sigma_0\in \Pi(\mu_2,\nu_1)$ with $W_p(\sigma_1,\sigma_0)\le W_p(\mu_1,\mu_2)$,
and then we take a measure $\sigma_2\in \Pi(\mu_2,\nu_2)$ with
$W_p(\sigma_0,\sigma_2)\le W_p(\nu_1,\nu_2)$.
\end{proof}

We recall that a completely regular space is called
  sequentially Prohorov if every sequence of probability
  Radon measures on this space weakly converging to a  Radon measure is uniformly tight.
  For example, complete metric spaces are sequentially Prohorov (see \cite{B07} or \cite{B18}
  about this property).

\begin{theorem}\label{maint2}
Let $X$ and $Y$ be
completely regular spaces. Suppose that  measures $\mu_n\in \mathcal{P}_r(X)$  converge weakly to a measure $\mu\in \mathcal{P}_r(X)$,
measures $\nu_n\in \mathcal{P}_r(Y)$  converge weakly to a measure $\nu \in \mathcal{P}_r(Y)$,
both sequences are uniformly tight {\rm(}which holds automatically if both spaces are sequentially Prohorov{\rm)},
and continuous functions
$h_n\colon X\times Y\to [0,+\infty)$ converge to a function $h\colon X\times Y\to [0,+\infty)$
uniformly on compact sets. Suppose also that there are
nonnegative Borel functions $a_n\in L^1(\mu_n)$ and $b_n\in L^1(\nu_n)$ such that
\begin{equation}\label{main-bound}
h_n(x,y)\le a_n(x)+b_n(y),
\quad
\lim\limits_{R\to+\infty} \sup_n \biggl(\int_{\{a_n\ge R\}} a_n\, d\mu_n+ \int_{\{b_n\ge R\}} b_n\, d\nu_n\biggr)=0.
\end{equation}
Then
$$
K_h(\mu,\nu)=\lim\limits_{n\to\infty} K_{h_n}(\mu_n,\nu_n).
$$
In particular, this is true if
$\sup_n (\|a_n\|_{L^p(\mu_n)}+\|b_n\|_{L^p(\nu_n)})<\infty$
for some $p>1$.
\end{theorem}
\begin{proof}
We first consider the case where $h_n\equiv h\le 1$.
Since $h$ is continuous on compact sets,
there are optimal measures $\sigma_n\in \Pi(\mu_n,\nu_n)$ for the function $h$ (see \cite[comments after Theorem~1.2.1]{BK}).
By our assumption, the sequences of measures $\mu_n$ and $\nu_n$ are uniformly tight, which  implies
the uniform tightness of the  sequence of measures $\sigma_n$.
This sequence contains some measure $\sigma$ in its closure  in the  weak topology, moreover, $\sigma\in \Pi(\mu,\nu)$.
The integral of $h$ with respect to the measure $\sigma$ is a limit point of the integrals against the measures $\sigma_n$, that is, of the
numbers $K_h(\mu_n,\nu_n)$. Hence $K_h(\mu,\nu)\le \liminf_{n\to\infty} K_h(\mu_n,\nu_n)$.

We show that $K_h(\mu,\nu)\ge \limsup_{n\to\infty} K_h(\mu_n,\nu_n)$.
Let $\varepsilon>0$.  Passing to a subsequence, we can assume that
$K_h(\mu_n,\nu_n)\to \limsup_{n\to\infty} K_h(\mu_n,\nu_n)$.
Let us take  compact sets $S_1\subset X$, $S_2\subset Y$ such that
$$
\mu_n(X\backslash S_1)+\nu_n(Y\backslash S_2) <\varepsilon \quad \forall\, n.
$$
The set  $S=S_1\times S_2$ is compact in $X\times Y$ and
\begin{equation}\label{ti1}
\zeta((X\times Y)\backslash S)<\varepsilon \quad \forall\, \zeta\in \Pi(\mu_n, \nu_n), \forall\, n.
\end{equation}
We can  take  bounded continuous pseudometrics  $\alpha$ and $\beta$
 on the spaces $X$ and $Y$ and a function $g\colon X\times Y\to [0,1]$ such that
$$
|g(x,y)-g(x_1,y_1)|\le \alpha(x,x_1)+\beta(y,y_1) \quad \forall\, x,x_1\in X,\, y,y_1\in Y,
$$
$$
g(x,y)=h(x,y) \quad \forall\, (x,y)\in S.
$$
There are several known ways to do this. One is this. Let us embed our spaces $X$ and $Y$ homeomorphically
into locally convex spaces $X_1$ and~$Y_1$ (or just equip them with uniformities generating the original topologies).
Then the function $h$ is uniformly continuous on~$S$. It
 has a uniformly continuous extension $g\colon X_1\times Y_1\to [0,1]$ (this is true for any subsets, see \cite{Kat},
 but for a compact subset it suffices to take a continuous extension to the Stone-\v{C}ech compactification
 of $X_1\times Y_1$).
The function
$$
\alpha(x_1,x_2)=\sup_y |g(x_1,y)-g(x_2,y)|
$$
is a uniformly continuous pseudometric on~$X$, and the
function
$$
\beta(y_1,y_2)=\sup_x |g(x,y_1)-g(x,y_2)|
$$
is a uniformly continuous pseudometric on~$Y$. In addition,
$$
|g(x,y)-g(x_1,y_1)|\le |g(x,y)-g(x_1,y)|+|g(x_1,y)-g(x_1,y_1)|,
$$
so $|g(x,y)-g(x_1,y_1)|\le \alpha(x,x_1)+\beta(y,y_1)$.

Let us consider the Kantorovich pseudometrics $d_{K,\alpha}$ and $d_{K,\beta}$
 on $X$ and $Y$ generated by these pseudometrics $\alpha$ and $\beta$.
For all $n$ sufficiently large we have
$$
d_{K,\alpha}(\mu_n,\mu)+d_{K,\beta}(\nu_n,\nu)< \varepsilon.
$$
By the theorem proved above
 there exist measures $\zeta_n\in \Pi(\mu_n,\nu_n)$ such that
$$
d_{K, \alpha\oplus\beta}(\sigma,\zeta_n)< \varepsilon.
$$
Therefore, by (\ref{ti1}) one has
$$
\int g\, d\sigma \ge \int g\, d\zeta_n-\varepsilon
\ge \int h\, d\zeta_n  - 3\varepsilon \ge K_h(\mu_n,\nu_n)-3\varepsilon.
$$
Since $\varepsilon$ was arbitrary, we obtain the estimate
$$
K_h(\mu,\nu)\ge \limsup_{n\to\infty} K_h(\mu_n,\nu_n).
$$

We now consider the case of different $h_n$, but with a common bound $h_n\le 1$.
Let $\varepsilon>0$. Using the same compact set $S$ as above, we find  $N$ such that
$|h_n-h|\le \varepsilon$ for all $(x,y)\in S$ and $n\ge N$. Then
$$
|K_h(\mu_n,\nu_n)- K_{h_n}(\mu_n,\nu_n)|\le 2\varepsilon,
$$
because by (\ref{ti1}) for every measure $\eta\in \Pi(\mu_n,\nu_n)$ we have the estimate
$$
\eta((X\times Y)\backslash S)<\varepsilon,
$$
so the difference between the integrals of $h_n$ and $h$ with respect to the  measure $\eta$
does not exceed $2\varepsilon$. This yields our assertion in the  case of uniformly
bounded $h_n$. The general case  reduces easily to this case, because for the functions $\min (h_n,R)$
we have
\begin{align*}
\int [h_n -\min (h_n,R)] \, d\eta &\le  \int h_n I_{\{h_n\ge R\}}\, d\eta\le
\int [2a_n I_{\{a_n\ge R/2\}} +2b_n I_{\{b_n\ge R/2\}}]\, d\eta
\\
&=2 \int_{\{a_n\ge R/2\}} a_n\, d\mu_n +
2\int_{\{b_n\ge R/2\}} b_n\, d\nu_n
\end{align*}
for all measures $\eta\in\Pi(\mu_n,\nu_n)$.
\end{proof}

\begin{corollary}\label{conv-pl}
If in the situation of the previous theorem optimal plans $\sigma_n$ and $\sigma$ for the triples
$(\mu_n,\nu_n,h_n)$ and $(\mu,\nu,h)$ are unique, then the measures $\sigma_n$ converge weakly to $\sigma$.
\end{corollary}
\begin{proof}
The sequence of measures $\sigma_{n}$ is uniformly tight by the uniform tightness of its marginals.
Hence it has a weakly convergent subnet. Note that
if its subnet $\{\sigma_\alpha\}$ converges weakly to a measure~$\sigma_0$, then $\sigma_0$ is optimal for~$h$.
Indeed, we show that
\begin{equation}\label{b1}
\int h\, d\sigma \le \lim\limits_{n\to\infty} K_{h_n}(\mu_{n},\nu_{n}).
\end{equation}
Otherwise there are numbers $\varepsilon>0$ and $R>1$ such that
$$
\int \min(h,R)\, d\sigma > \lim\limits_{n\to\infty} K_{h_n}(\mu_{n},\nu_{n})+\varepsilon.
$$
By weak convergence
$$
\int \min(h,R)\, d\sigma = \lim\limits_{\alpha} \int \min(h,R)\, d\sigma_{\alpha}.
$$
Hence we can assume that
$$
\int \min(h,R)\, d\sigma_{\alpha}\ge \lim\limits_{n\to\infty} K_{h_n}(\mu_{n},\nu_{n})+\varepsilon.
$$
Then we can find infinitely many indices $n$ such that
$$
\int \min(h,R)\, d\sigma_{n}\ge \int h_n\, d\sigma_n +\varepsilon/2.
$$
However, it is clear that for all $n$ large enough
$$
\int \min(h,R)\, d\sigma_{n}\le \int \min(h_n,R)\, d\sigma_n +\varepsilon/4,
$$
because the measures $\sigma_n$ are uniformly tight and the functions $\min(h_n,R)$ converge
to $\min(h,R)$ uniformly on compact sets.

The right-hand of (\ref{b1}) equals $K_{h}(\mu,\nu)$ by the theorem.
Hence $\sigma_0$ is optimal and then
$\sigma_0=\sigma$ by uniqueness. Therefore, the measures $\sigma_{n}$ converge weakly to $\sigma$.
\end{proof}

For uncountable families of measures and functions an analog of (\ref{main-bound}) reads as
\begin{equation}\label{mainboundt}
h_t(x,y)\le a_t(x)+b_t(y),
\quad
\lim\limits_{R\to+\infty} \sup_t \biggl(\int_{\{a_t\ge R\}} a_t\, d\mu_t+ \int_{\{b_t\ge R\}} b_t\, d\nu_t\biggr)=0.
\end{equation}
In particular, this is true if
$$
\sup_t [\|a_t\|_{L^p(\mu_t)}+\|a_t\|_{L^p(\nu_t)}]<\infty
$$
for some $p>1$.

\begin{corollary}
Let $X$ and $Y$ be  sequentially Prohorov
completely regular spaces and let $T$ be a topological space.
Suppose that  the mappings
$$
t\mapsto \mu_t\in \mathcal{P}_r(X)
\quad \hbox{and}\quad
t\mapsto \nu_t\in \mathcal{P}_r(Y)
$$
 are sequentially continuous and
$(t,x,y)\mapsto h_t(x,y)$, $T\times X\times Y\to [0,+\infty)$ is a continuous function.
Suppose also that there exist
nonnegative Borel functions \mbox{$a_t\in L^1(\mu_t)$} and $b_t\in L^1(\nu_t)$ such that {\rm(\ref{mainboundt})} holds.
Then the function $t\mapsto K_{h_t}(\mu_t,\nu_t)$ is sequentially continuous.
\end{corollary}

For the proof it suffices to use that $h_{t_n}(x,y)\to h_t(x,y)$ uniformly on compact sets if $t_n\to t$.

Obviously, the Prohorov property can be replaced by the uniform tightness of the families $\{\mu_t\}$ and $\{\nu_t\}$.

Suppose now that $X, Y, T$ are metric spaces, $X$ and $Y$ are complete, and for every $t\in T$
we are given measures $\mu_t\in \mathcal{P}_r(X)$ and $\nu_t\in \mathcal{P}_r(Y)$ such that
the mappings $t\mapsto \mu_t$ and $t\mapsto \nu_t$ are continuous in the weak topology
(which is equivalent to the continuity in the Kantorovich--Rubinshtein metric). Suppose also
that there is a continuous nonnegative function $(t,x,y)\mapsto h_t(x,y)$.
The question arises whether it is  possible to select an  optimal plan continuously depending on the parameter $t$.
It turns out that  such a choice is not always possible, as is shown by the examples below.
However, the situation improves for approximate optimal plans or in case of unique optimal plans.
Given $\varepsilon>0$, a measure $\sigma\in \Pi(\mu,\nu)$ will be called
$\varepsilon$-optimal for the cost function $h$ if
$$
\int f\, d\sigma \le K_h(\mu,\nu)+\varepsilon.
$$

\begin{theorem}\label{t-appr}
Suppose that for every $t$ there exist nonnegative Borel functions \mbox{$a_t\in L^1(\mu_t)$} and $b_t\in L^1(\nu_t)$
such that {\rm(\ref{mainboundt})} holds.
Then one can select $\varepsilon$-optimal  measures $\sigma_t^\varepsilon\in \Pi(\mu_t,\nu_t)$ for the cost functions $h_t$
such that they will be continuous in $t$ in the weak topology for every fixed $\varepsilon>0$.

If for every $t$ there is a unique optimal plan~$\sigma_t$, then it is continuous
in~$t$.
\end{theorem}
\begin{proof}
The assumption implies the inclusion
$h_t\in L^1(\sigma)$ for all $\sigma\in \Pi(\mu_t,\nu_t)$.
The set
$$
M_t=\biggl\{\sigma\in \Pi(\mu_t,\nu_t)\colon \int h_t\, d\sigma=K_{h_t}(\mu_t,\nu_t)\biggr\}
$$
is convex and compact in $\mathcal{P}_r(X\times Y)$. Indeed, the convexity is obvious
and the compactness follows from the fact that $M_t$ is closed in the compact set $\Pi(\mu_t,\nu_t)$,
which is verified in the following way. Suppose that measures $\pi_n\in M_t$ converge weakly  to a measure~$\pi$.
For convergence of the integrals of the continuous function $h_t$ against the measures $\pi_n$ to the integral against the measure $\pi$
it suffices to verify (see \cite[Theorem~2.7.1]{B18}) that
$$
\lim\limits_{R\to\infty} \sup_n \int_{\{h_t\ge R\}} h_t\, d\pi_n =0.
$$
Since $\pi_n\in \Pi(\mu_t,\nu_t)$, this equality follows from the estimate
\begin{multline*}
\int_{\{h_t\ge R\}} h_t\, d\pi_n\le \int_{X\times Y} [2a_tI_{\{a_t\ge R/2\}} + 2b_tI_{\{b_t\ge R/2\}}] \, d\pi_n
\\
=2\int_{\{a_t\ge R/2\}} a_t\, d\mu_t+ 2\int_{\{b_t\ge R/2\}} b_t\, d\nu_t.
\end{multline*}
Set
$$
\Psi(t) = \biggl\{\pi \in \Pi(\mu_t, \nu_t)\colon
 \int h_t\, d\pi \le K_{h_t}(\mu_t, \nu_t) + \varepsilon
 \biggr\}.
 $$
In order to find continuous selections we verify the hypotheses of the classical Michael selection theorem
applied to the multivalued mapping $t\mapsto \Psi_t$ with convex  compact values
in the complete metrizable subset $\mathcal{P}_r(X\times Y)$ of the locally convex
space $\mathcal{M}_r(X\times Y)$ (see, e.g., \cite{AF} or \cite{RS}). We have to show that
for every open set $U$ in $\mathcal{P}_r(X\times Y)$ the set
$$
W=\{t\in T\colon \Psi_t\cap U\not=\emptyset \}
$$
is open in $T$. Let $w\in W$. Then there exists a measure $\sigma\in \Psi_{w}\cap U$, that is,
$\sigma \in \Pi(\mu_{w}, \nu_{w})$ and
$$
\int h_w\, d\sigma \le K_{h_w}(\mu_w, \nu_w) + \varepsilon.
$$
We can  assume that this inequality is strict, since otherwise we can
take a convex linear combination with an
optimal measure $\sigma_0\in M_w$, which for small $\alpha>0$ will give a  measure $(1-\alpha) \sigma + \alpha \sigma_0\in U$
with a smaller integral of the function~$h_w$. Thus,
$$
\int h_w\, d\sigma \le K_{h_w}(\mu_w, \nu_w) + \varepsilon -\delta, \quad \hbox{where} \ \delta>0.
$$
Let us show that there exists a  neighborhood $V$ of the point $w$ such that for every $v\in V$ there is
a measure $\sigma_v \in \Pi(\mu_v, \nu_v)$ for which
$$
\int h_v\, d\sigma \le K_{h_v}(\mu_v, \nu_v) + \varepsilon .
$$
Indeed, otherwise there exists a  sequence $w_n\to w$ such that
\begin{equation}\label{eq-c}
\int h_{w_n}\, d\zeta > K_{h_{w_n}}(\mu_{w_n}, \nu_{w_n}) + \varepsilon \quad \forall\, \zeta\in \Pi(\mu_{w_n}, \nu_{w_n}).
\end{equation}
It follows from (\ref{mainboundt}) that there is $R>1$ such that
\begin{equation}\label{srez}
\int [h_t -\min(h_t,R)]\, d\zeta \le \delta/8 \quad \quad \forall\, t\in T, \, \forall \, \zeta\in \Pi(\mu_t,\nu_t).
\end{equation}
Next, the weak continuity of $\mu_t$ and $\nu_t$ in $t$ implies  that the sequences $\{\mu_{w_n}\}$ and $\{\nu_{w_n}\}$
are uniformly tight. This implies the uniform tightness of the union of the sets $\Pi(\mu_{w_n}, \nu_{w_n})$. Hence
there is a   compact set $S\subset X\times Y$ such that
\begin{equation}\label{plot}
(\zeta+\sigma)((X\times Y)\backslash S)< \delta (32R)^{-1} \quad \forall\, n, \, \forall\, \zeta\in \Pi(\mu_{w_n}, \nu_{w_n}).
\end{equation}
Set
$$
g_t=\min(h_t,R).
$$
By the continuity of $h$ on the compact set $(\{w_n\}\cup \{w\})\times S$
there exists a Lipschitz function $(t,x,y)\to L_t(x,y)$ on $T\times X\times Y$ with values in $[0,R]$ for which
\begin{equation}\label{prib}
|g_{w_n}(x,y)-L_{w_n}(x,y)|\le \delta /32 \quad \forall\, n\ge 1, \, (x,y)\in S.
\end{equation}
Let $L>0$ be the Lipschitz constant of this function.
Since the compact set $\Pi(\mu_t, \nu_t)$ depends on $t$  continuously in the Hausdorff metric,
for all $n$ sufficiently large we  have
$$
H_K(\Pi(\mu_{w_n}, \nu_{w_n}),\Pi(\mu_{w}, \nu_{w}))\le \delta/(16 L).
$$
Hence for every such $n$ there is a measure $\zeta_n\in \Pi(\mu_{w_n}, \nu_{w_n})$ satisfying the estimate
$$
d_K(\sigma,\zeta_n)\le \delta (16L)^{-1},
$$
which yields the inequality
$$
\biggl|\int L_w\, d\sigma - \int L_w\, d\zeta_n\biggr| \le \delta/16.
$$
Since
$$
\sup_{(x,y)\in S} |L_w(x,y)-L_{w_n}(x,y)|\to 0,
\quad L_t\le R, \quad
(\zeta_n+\sigma)((X\times Y)\backslash S)< \delta (32R)^{-1},
$$
for all $n$ sufficiently large we  have
$$
\biggl|\int L_w\, d\sigma - \int L_{w_n}\, d\zeta_n\biggr| \le \delta/8.
$$
Therefore, by (\ref{plot}) and (\ref{prib}) we obtain
$$
\biggl|\int g_w\, d\sigma - \int g_{w_n}\, d\zeta_n\biggr| \le \delta/4,
$$
which by (\ref{srez}) gives the estimate
$$
\biggl|\int h_w\, d\sigma - \int h_{w_n}\, d\zeta_n\biggr| \le \delta/2.
$$
Since $K_{h_{w_n}}(\mu_{w_n},\nu_{w_n})\to K_{h_{w}}(\mu_{w},\nu_{w})$,
we arrive at the contradiction with (\ref{eq-c}).
Thus, the set $\{t\colon \Psi(t) \cap U \not= \emptyset\}$ open.

Finally, in case of unique optimal plans the result follows from Corollary~\ref{conv-pl}. 
\end{proof}

In the general case there is no continuous selection of optimal plans.

\begin{example}
{\rm
Let $\mu$ and $\nu$ be  Borel probability measures on a complete separable metric
space $X$ such that there are at least  two continuous mappings $f, g\colon X\to X$
taking $\mu$ to $\nu$ for which the images of $\mu$ under the mappings
$x\mapsto (x,f(x))$ and $x\mapsto (x,g(x))$ are different.
For example, for $\mu$ and $\nu$ one can take  Lebesgue measure on $[0,1]$.
Let us take for $T$ the sequence of points $1/n$ and~$0$.
Set
$$
h_0=0, \ h_t(x,y)=t |y-f(x)|^2 \quad \hbox{if } \ t=(2n-1)^{-1},
$$
$$
h_t(x,y)=t |y-g(x)|^2 \quad \hbox{if }\ t=(2n)^{-1}.
$$
 Let $\mu_t=\mu$, $\nu_t=\nu$.
For $t=0$ all measures in $\Pi(\mu,\nu)$ are optimal  for~$h_0$. For $t=(2n-1)^{-1}$ a unique optimal
measure for the cost function $h_t$ is the image of $\mu$ under the mapping $x\mapsto (x,f(x))$,
for $t=(2n)^{-1}$ a unique optimal measure for the cost function $h_t$ is the image of $\mu$ under the
 mapping $x\mapsto (x,g(x))$.
Then the indicated  optimal measures have no limit as $t\to 0$. For example, if
$\mu=\nu$ is  Lebesgue measure on $[0,1]$, then we can take $f(x)=x$, $g(x)=1-x$.
Here is another explicit example with the same measures
on~$[0, 1]$:
$$
h_t(x, y) = \min(|x-y|, |x+y-1| + t), \ t \ge 0,
h_t(x, y) = \min(|x-y| - t, |x+y-1|),\  t < 0.
$$
Then for $t>0$ the optimal plan is concentrated on the diagonal $x=y$, for
$t<0$ the optimal plan is concentrated on the diagonal $x+y=1$,
in all cases it is the normalized Lebesgue measure. For $t=0$ there is no uniqueness:
there are many optimal plans concentrated on the union of the diagonals,
some of them are not generated by Monge mappings. However, the left and right limits of optimal plans
are generated by Monge mappings.
}\end{example}

It is worth noting that a related, but not equivalent problem was considered in \cite{AP} and \cite{DPLS}
(see also \cite{V1}), where convergence in measure of Monge optimal mappings $T_\varepsilon\to T_0$ as $\varepsilon\to 0$
was shown for some special cost functions $h_\varepsilon$.  In our situation a relevant question would be
about some form of continuity of Monge optimal mappings $T_t$ in the case where $\mu_t$ does not depend on $t$
or all measures $\mu_t$ are absolutely continuous with respect to some reference measure $\lambda$.
Of course, in case of non-uniqueness of $T_t$ the question is about continuous selections of optimal mappings.
Here is a general result in the spirit of the cited papers in the case where optimal plans are generated
by mappings. The particular situation considered in the cited papers (see also \cite[Theorem~2.53]{V1})
deals with constant marginals and cost functions $h_t(x,y)=|x-y|+t |x-y|^2$ or a more general family of functions with the property that
$h_t(x,y)=|x-y|+t \varphi(|x-y|)+o(t)$ with a strictly convex function~$\varphi$. However, this special structure is needed
to guarantee the existence and uniqueness of Monge solutions. In the next abstract result the existence of ``Monge mappings'' is part
of the hypotheses.

Recall that convergence in measure for mappings to completely regular spaces is defined as follows. Let $\{d_\tau\}$ be a
family of pseudometrics defining the topology of a completely regular space~$X$ and let $\mu\in\mathcal{P}_r(X)$.
Borel mappings $T_n\colon X\to X$ converge in measure $\mu$ to a Borel mapping $T\colon X\to X$ if, for each~$\tau$
and each $\delta>0$,
one has $\mu(x\colon d_\tau(T_n(x),T(x))\ge \delta)\to 0$ as $n\to \infty$.

\begin{proposition}
Let $X$ be a completely regular space,
$\mu_n, \nu_n\in\mathcal{P}_r(X)$ for $n\in \mathbb{Z}^+$, $\mu_n\to\mu_0$ in variation, and
the measures $\mu_0$ and $\nu_0$ are concentrated on countable unions of metrizable compact sets
{\rm(}which holds automatically if $X$ is a Souslin space{\rm)}.
Suppose that
there are Borel mappings $T_n\colon X\to X$ such that the measures $\sigma_n$ equal to the images of the measures $\mu_n$
under the mappings $x\mapsto (x,T_n(x))$ converge weakly to~$\sigma_0$.
 Then the mappings $T_n$ converge to $T_0$ in measure with respect to~$\mu_0$.
  \end{proposition}
\begin{proof}
Let $\psi$ be a continuous pseudometric on~$X$.
We can assume that $0\le \psi\le 1$.
  Suppose first that $T_{0}$ is continuous. Then
 the function $\psi(T_{0}(x),y)$ is continuous on~$X^2$, hence
 due to weak convergence $\sigma_n\to \sigma_{0}$ we have
\begin{multline*}
\int \psi(T_0(x),T_n(x)) \, \mu_n(dx)=\int \psi(T_{0}(x),y)\, \sigma_{n}(dx\, dy)
\\
 \to
\int \psi(T_{0}(x),y)\, \sigma_{0}(dx\, dy)=\int \psi(T_0(x),T_{0}(x)) \, \mu_0(dx)= 0.
\end{multline*}
Then by convergence in variation
$$
\int \psi(T_0(x),T_n(x)) \, \mu_0(dx)\to 0.
$$
Hence $T_n\to T_{0}$ in measure $\mu_0$.

In the general case, we can embed $X$ homeomorphically into a suitable power  of the real line and assume that $X=\mathbb{R}^\tau$.
Given $\varepsilon>0$,
we can find a compact set $K$ with $\mu_0(K)>1-\varepsilon$ on which $T_{0}$ is continuous.
To this end, we find a metrizable compact set $Q$ with $\nu_0(Q)>1-\varepsilon$ and in the Borel set $T_0^{-1}(Q)$ we find
a metrizable compact set $Q_1$ with $\mu_0(Q_1)>1-\varepsilon$, which is possible, because $\mu_0(T_0^{-1}(Q))=\nu_0(Q)>1-\varepsilon$.
It remains to appy Luzin's theorem to the Borel mapping $T_0$ between metrizable compact sets $Q_1$ and~$Q$.
Actually, in order to apply a generalization of Luzin's theorem, it suffices that only $\nu_0$ be concentrated
on metrizable compact sets (see \cite[Section 7.14(ix)]{B07}).
Next, there is a continuous
mapping $S$ that coincides with $T_{0}$ on~$K$ (here we use that $X=\mathbb{R}^\tau$, so it suffices
to extend the components of~$T_0$). Repeating the same reasoning with
the function $\psi(S(x),y)$, we obtain that the integral of $\psi(S(x),T_n(x))$ against $\mu_0$
tends to the integral of $\psi(S(x), T_0(x))$, which is estimated by~$\varepsilon$, since $\psi(S(x),T_{0}(x))=0$
on~$K$. This yields convergence in measure in the general case.
\end{proof}

Note that convergence of $T_n$ to $T_0$ in measure $\mu_0$ is also sufficient for weak convergence
of~$\sigma_n$.

\begin{corollary}\label{conv-meas}
Suppose that in the situation of Corollary \ref{conv-pl} the space $X=Y$ is metric and the unique optimal plans
$\sigma_n$ are generated by unique Monge optimal mappings~$T_n$. If convergence $\mu_n\to\mu_0$ holds in variation, then
the mappings $T_n$ converge to $T_0$ in measure~$\mu_0$.
\end{corollary}

\begin{remark}
\rm
As shown above, in the situation when we deal with optimal transportation problems for triples $(\mu_n,\nu_n,h_n)$,
the hypothesis that the plans $\sigma_n$ converge weakly to $\sigma_0$  is fulfilled if optimal measures are unique, the measures $\nu_n$
converge weakly to $\nu_0$ and are uniformly tight, and the functions $h_n$
are continuous and converge to $h_0$ uniformly on compact sets.

If we do not assume weak convergence of plans, but $\{\nu_n\}$ converges weakly and is uniformly tight,
then the conclusion still holds for a subsequence in $\{T_n\}$ picked such that
the corresponding plans converge, provided it is known that all optimal plans for $h_0$
(or at least those in the closure of the considered subsequence of plans) are generated by Monge mappings.

The particular case of constant marginals $\mu$ and $\nu$ is not much simpler, because anyway we have to ensure
convergence of plans and need existence of Monge mappings.
It would be interesting to study approximate Monge solutions depending continuously on a parameter.
To this end, one could analyze the constructions in \cite{BKP} and \cite{P}.
\end{remark}

This research is supported by the
Russian Foundation for Basic Research Grant 20-01-00432 and
Moscow Center of Fundamental and Applied Mathematics.


\begin{thebibliography}{99}

\bibitem{AG}
L. Ambrosio, N. Gigli,
A user's guide to optimal transport,
Lecture Notes in Math.  2062 (2013), 1--155.

\bibitem{AP}
  L. Ambrosio,  A. Pratelli,
Existence and stability results in the $L^1$
theory of optimal transportation,
In: Optimal transportation and applications (Martina Franca, 2001),
 Lecture Notes in Math., V. 1813,  pp.~123--160, Springer, 2003.

\bibitem{AF}
J.-P. Aubin,  H. Frankowska,
Set-valued Analysis,  Birkh\"auser Boston, Boston, 1990.

\bibitem{B07}
 V.I. Bogachev,
Measure Theory, vols.~1,~2, Springer, Berlin, 2007.

\bibitem{B18}
V.I. Bogachev, Weak Convergence of Measures, Amer. Math. Soc., Providence,
Rhode Island, 2018.

\bibitem{BDM}
V.I. Bogachev, A.N. Doledenok, I.I. Malofeev,
The Kantorovich problem with a parameter and density constraints,
Mathematical Notes 110 (6) (2021), 149--153.

\bibitem{BKP}
V.I. Bogachev,  A.N. Kalinin, S.N. Popova,  On the equality of values in the Monge and Kantorovich problems,
  Zap. Nauchn. Sem. S.-Peterburg. Otdel. Mat. Inst. Steklov. (POMI) 457 (2017), 53--73 (Russian);
 translation in J. Math. Sci. (N.Y.) 238 (4) (2019), 377--389.

\bibitem{BK}
 V.I. Bogachev, A.V. Kolesnikov,
 The Monge--Kantorovich problem: achievements, connections, and prospects,
Uspekhi Matem. Nauk 67 (5) (2012), 3--110 (in Russian); English transl.:
Russian Math. Surveys 67 (5) (2012), 785--890.

\bibitem{BM}
V.I. Bogachev, I.I. Malofeev,
 Kantorovich problems and conditional measures depending on a parameter,
 J. Math. Anal. Appl. 486 (1) (2020), 1--30.

\bibitem{BF}
 B. Bonnet, H. Frankowska,
Differential inclusions in Wasserstein spaces: the Cauchy-Lipschitz framework,
J. Differential Equations 271 (2021), 594--637.

\bibitem{C21}
 C. Clason,  D.A. Lorenz,  H. Mahler, B. Wirth,
 Entropic regularization of continuous optimal transport problems,
J. Math. Anal. Appl. 494 (1) (2021), Paper No. 124432, 22 pp.

 \bibitem{DePR}
  J. Dedecker,  C. Prieur, P. Raynaud De Fitte,
 Parametrized Kantorovich--Rubin\v{s}tein theorem
 and application to the coupling of random variables,
 In: Dependence in probability and statistics,
 pp.~105--121, Lect. Notes Stat., V.~187, Springer, New York, 2006.

 \bibitem{DPLS}
 L. De Pascale, J. Louet, F. Santambrogio,
 The Monge problem with vanishing gradient penalization: vortices and asymptotic profile,
  J. Math. Pures Appl. (9), 106 (2) (2016), 237--279.

 \bibitem{Kat}
M. Kat\v{e}tov, On real-valued functions in topological spaces, Fund. Math. 38
(1951), 85--91; Correction: 40 (1953), 203--205.

 \bibitem{KNS}
 S. Kuksin,   V. Nersesyan, A. Shirikyan,
Exponential mixing for a class of dissipative PDEs with bounded degenerate noise,
Geom. Funct. Anal. (GAFA) 30 (1) (2020), 126--187.

 \bibitem{LMM}
  D.A. Lorenz,  P. Manns, C. Meyer,
  Quadratically regularized optimal transport, Appl. Math. Optim. 83 (3) (2021), 1919--1949.

 \bibitem{Mal}
 I.I. Malofeev,  Measurable dependence of conditional measures on a parameter,
 Dokl. Akad. Nauk 470 (1) (2016), 13--17 (in Russian);
 English transl.: Dokl. Math. 94 (2) (2016), 493--497.

 \bibitem{P}
A. Pratelli,
On the equality between Monge's infimum and Kantorovich's
 minimum in optimal mass transportation,
Ann. Inst. H.~Poincar\'e~(B) Probab. Statist. 43 (1) (2007), 1--13.

  \bibitem{RR}
 S.T. Rachev, L. R{\"u}schendorf,
 Mass Transportation Problems, vols.~I,~II, Springer, New York, 1998.

\bibitem{RS}
 D. Repov\v{s}, P.V. Semenov,
Continuous Selections and Multivalued Mappings,
Kluwer Acad. Publ., Dordrech -- Boston -- London, 1998.

 \bibitem{Sant}
 Santambrogio F.
 Optimal Transport for Applied Mathematicians, Birk\-h\"au\-ser/\-Springer, Cham,
 2015.

\bibitem{V1}
 C. Villani,   Topics in Optimal Transportation,
Amer. Math. Soc., Providence, Rhode Island, 2003.

\bibitem{V2}
C. Villani,  Optimal Transport, Old and New, Springer, New York, 2009.

\bibitem{Z}
X. Zhang,  Stochastic Monge--Kantorovich problem and its duality,
           Stochastics 85 (1) (2013), 71--84.

\end{thebibliography}
\end{document}